\documentclass[reqno,a4paper,12pt]{amsart}
\usepackage{latexsym, amsmath, amssymb, amsthm, a4}

\usepackage{amsfonts, bbold, bbm, marvosym}

\usepackage{mathrsfs}

\usepackage[mathscr]{eucal}
\font\sc=rsfs10 at 12pt

\renewcommand{\a}{\alpha}

\newcommand{\g}{\gamma}
\newcommand{\G}{\Gamma}
\renewcommand{\d}{\delta}
\newcommand{\D}{\Delta}
\newcommand{\e}{\epsilon}
\newcommand{\ve}{\varepsilon}
\newcommand{\z}{\zeta}

\newcommand{\vt}{\vartheta}
\renewcommand{\i}{\iota}
\renewcommand{\k}{\kappa}

\renewcommand{\l}{\lambda}

\newcommand{\m}{\mu}
\newcommand{\n}{\nu}
\newcommand{\x}{\xi}

\renewcommand{\r}{\rho}
\newcommand{\s}{\sigma}
\newcommand{\Si}{\Sigma}
\newcommand{\vs}{\varsigma}
\renewcommand{\t}{\tau}

\newcommand{\h}{\chi}
\newcommand{\vp}{\varpi}

\renewcommand{\o}{\omega}
\renewcommand{\O}{\Omega}

\newcommand{\C}{{\mathbb C}}

\newcommand{\R}{{\mathbb R}}

\newcommand{\bm}{\pmb{\m}}

\newcommand{\ab}{{\mathbf a}}

\newcommand{\eb}{{\mathbf e}}

\newcommand{\kb}{{\mathbf k}}

\newcommand{\pb}{{\mathbf p}}
\newcommand{\qb}{{\mathbf q}}
\newcommand{\rb}{{\mathbf r}}

\newcommand{\vb}{{\mathbf v}}

\newcommand{\Ab}{{\mathbf A}}

\newcommand{\Bb}{{\mathbf B}}

\newcommand{\Cb}{{\mathbf C}}

\newcommand{\Fb}{{\mathbf F}}

\newcommand{\Hb}{{\mathbf H}}
\newcommand{\Ib}{{\mathbf I}}
\newcommand{\Jb}{{\mathbf J}}

\newcommand{\Nb}{{\mathbf N}}

\newcommand{\Qb}{{\mathbf Q}}

\newcommand{\Vb}{{\mathbf V}}

\newcommand{\Xb}{{\mathbf X}}

\newcommand{\Zb}{{\mathbf Z}}

\newcommand{\km}{\mathbbm{k}}
\newcommand{\aF}{\mathfrak a}
\newcommand{\AF}{\mathfrak A}
\newcommand{\bF}{\mathfrak b}
\newcommand{\BF}{\mathfrak B}

\newcommand{\cF}{\mathfrak c}
\newcommand{\CF}{\mathfrak C}
\newcommand{\dF}{\mathfrak d}

\newcommand{\kF}{\mathfrak k}
\newcommand{\KF}{\mathfrak K}

\newcommand{\LF}{\mathfrak L}
\newcommand{\mF}{\mathfrak m}
\newcommand{\MF}{\mathfrak M}

\newcommand{\qF}{\mathfrak q}
\newcommand{\QF}{\mathfrak Q}

\newcommand{\rF}{\mathfrak r}
\newcommand{\RF}{\mathfrak R}

\newcommand{\TF}{\mathfrak T}

\newcommand{\WF}{\mathfrak W}

\newcommand{\VF}{\mathfrak V}

\newcommand{\Ht}{\mathrm{H}}
\newcommand{\Ac}{{\mathcal A}}
\newcommand{\Bc}{{\mathcal B}}

\newcommand{\Dc}{{\mathcal D}}
\newcommand{\Ec}{{\mathcal E}}
\newcommand{\Fc}{{\mathcal F}}

\newcommand{\Hc}{{\mathcal H}}

\newcommand{\Jc}{{\mathcal J}}

\newcommand{\Lc}{{\mathcal L}}

\newcommand{\Nc}{{\mathcal N}}
\newcommand{\Oc}{{\mathcal O}}

\newcommand{\Rc}{{\mathcal R}}

\newcommand{\Tc}{{\mathcal T}}
\newcommand{\Uc}{{\mathcal U}}

\newcommand{\Hs}{\sc\mbox{H}\hspace{1.0pt}}

\newcommand{\diag}{{\rm diag}\,}

\newcommand{\vol}{{\rm vol}\,}

\newcommand{\sub}{\footnotesize{\rm{sub}}\,}

\newcommand{\Q}{\QF}

\newcommand{\Hom}{\operatorname{Hom\,}}

\newcommand{\End}{\operatorname{End\,}}

\newcommand{\meas}{\operatorname{meas\,}}

\newcommand{\sme}{\sigma_{\operatorname{ess}}}

\renewcommand{\km}{\mathbbm{k}}

\newcommand{\bn}{\pmb{\nu}}

\newcommand{\beq}{\begin{equation}}
\newcommand{\eeq}{\end{equation}}

\numberwithin{equation}{section}
\numberwithin{figure}{section}

\newtheorem{thm}{Theorem}[section]

\newtheorem{lem}[thm]{Lemma}

\newtheorem{proposition}[thm]{Proposition}

\newtheorem{theorem}{Theorem}[section]

\newtheorem{Proposition}[theorem]{Proposition}

\theoremstyle{remark}
\newtheorem{remark}[theorem]{Remark}

\theoremstyle{definition}

\begin{document}

\title[Zero order operators]{Discrete spectrum of zero order pseudodifferential operators}

\author{Grigori Rozenblum }

\address{Chalmers Univ. of Technol., Sweden; The Euler Intern. Math. Institute and St.Petersburg State Univ.; Mathematics Center
Sirius Univ. of Sci. and Technol.
Sochi Russia}
\email{grigori@chalmers.se}

\thanks{The author was supported  by the grant of the Russian Fund of Basic Research 20-01-00451.}

\begin{abstract}
We study the rate of convergence of eigenvalues to the endpoints of essential spectrum for zero order pseudodifferential operators on a compact manifold.
\end{abstract}
\maketitle



\section{Introduction}\label{Intro}
 \subsection{Zero order operators and their spectrum}\label{SubSe.zero}Spectral properties of self-adjoint pseudodifferential operators on a compact manifold have been the topic of intensive study at least for the latest 70 years. For elliptic operators of positive order, thus having discrete spectrum, Weyl type asymptotics is known since long ago, with the second term existing under some geometric conditions. For operators of negative order, the compact ones, the eigenvalue asymptotics formula was established by M.Birman and M.Solomyak, see \cite{BS}, a more elementary proof, under certain regularity conditions was presented in \cite{Ponge}; a remainder estimate and, again, the second term in asymptotics, was investigated in \cite{Ivrii}, Sect. 11.8. Considerably less is known for pseudodifferential operators of zero order.  The  location of the essential spectrum $\sme(\AF)$ of such operator $\AF$ acting on a closed smooth manifold $\Xb$ is determined by the principal (order zero) symbol $\aF_0(x,\o)$, $\o=\x/|\x|,$ of $\AF.$ Namely, for operators acting on functions, $\sme(\AF)$ coincides with the set of values $\Rc(\aF_0)$ of the symbol $\aF_0$:
\begin{equation*}
    \sme(\AF)=\Rc(\aF_0)\equiv \bigcup_{(x,\o)\in S^*X}\{\aF_0(x,\o)\}.
\end{equation*}
For operators acting on sections of a vector bundle over $\Xb,$ the essential spectrum of $\AF$ coincides with the union of the spectra of the principal symbol $\aF_0(x,\o)$:
\begin{equation*}
    \sme(\AF)=\bigcup_{(x,\o)\in S^*\Xb}\s(\aF_0(x,\o))=\bigcup_{(x,\o)\in S^*\Xb}\bigcup_{\i}\{\bm_\i(x,\o)\},
\end{equation*}
where $\bm_\i(x,\o)$ are eigenvalues of the symbol $\aF_0(x,\o)$ (this simple but important fact was established in \cite{Adams} in the scalar case and in \cite{RPc} in the vector case). Recently, some applications required an analysis of the essential spectrum of such operators, see \cite{Zwor}, \cite{CdV1}, \cite{GalkoZwor}.

Less is known about the discrete spectrum of zero order operators. Examples in \cite{Tao} show that embedded eigenvalues can be present, even in  the one-dimensional case. As for non-embedded eigenvalues, they, of course, may converge only to the tips of the essential spectrum and a natural question arises about their rate of convergence to these tips. This kind of questions arizes, in particular, in the study of the Neumann-Poincare (NP, the double layer potential) operator $\KF$ for the three-dimensional elasticity. It is known that this operator is not compact, even for infinitely smooth data; it was found in \cite{AgrLame} that it is a zero order symmetrizable  pseudodifferential operator acting on smooth sections of a trivial three-dimensional vector bundle over the boundary $\Xb$ of a nice domain $\Dc$ in $\R^3.$ The principal symbol of this operator, a $3\times 3$ matrix, was found in  \cite{AgrLame} (see also \cite{3D}, \cite{MiRo}); it depends only on the Lam\'e constants of the material of the body but not on its geometry.  Therefore,
if the material is homogeneous, the principal symbol of the NP operator   has constant eigenvalues not depending on $(x,\o)\in S^*\Xb$ and therefore there are only finitely many (in fact, exactly 3) points of the essential spectrum of this operator. This means that the operator $\KF$ is  polynomially compact. Such kind of operators has been studied in \cite{RPc}, and it was found that the eigenvalues of a polynomially compact pseudodifferential operator converge to the points of the essential spectrum power-like. The rate of this convergence depends on the subprincipal symbol and,  possibly, in the degenerate case, on the some lower order symbols of the operator $\KF;$ the coefficients in the eigenvalue asymptotic formulas depend on the Lam\'e constants and geometrical characteristics of the surface $\Xb,$ see \cite{RPc}, \cite{R.NP.}.

For a nonhomogeneous material, the eigenvalues of the principal symbol of the NP operator are, generally, nonconstant and the essential spectrum may consist of several intervals in the real line and the isolated point $0$; for a body $\Dc$ with connected boundary $\Xb$ there are two such intervals, symmetrical with respect to the zero point.
The principal symbol of the NT operator equals
 \begin{equation}\label{princ}
 \kF_0(x,\o)=\km(x) \rF(\o)\equiv\frac{i}{\km(x)}\begin{pmatrix}
                                              0 & 0 & -\x_1 \\
                                              0 & 0 & -\x_2 \\
                                              \x_1 & \x_2 & 0 \\
                                            \end{pmatrix}
\end{equation}
in a special local co-ordinate system. The coefficient $\km(x)$ is determined by the Lame constants $\l(x),\m(x)$ at the boundary point $x\in \Xb,$ $\km(x)=\frac{\m(x)}{2(2\m(x)+\l(x))}$ while the matrix $\rF(\o)$,
$\o=\frac{\x}{|\x|}\in S_x^*(X)$ has eigenvalues $-1,0,1.$ Thus the essential spectrum of the NP operator consists of the point zero and the range of the functions $x\mapsto \km(x),$ $x\mapsto -\km(x),$ $x\in \Xb.$ If  $\km(x)$ is not constant on the surface $\Xb,$ these intervals are nontrivial; for a connected boundary they have the form  $\Jb_-=[-\km_+,-\km_-],$ $\Jb_{+}=[\km_-,\km_+],$ where $\km_-=\min_{x\in \Xb}\km(x),$ $\km_+=\max_{x\in \Xb}\km(x).$

In \cite{MiRo}, the case of one of these extremal points, say, $\km_-$, being a nondegenerate extremum of $\km(x)$ was considered, and certain estimates for the rate of convergence of eigenvalues of $\KF$, lying below $\km_-,$ to this point have been obtained. Asymptotic formulas for these eigenvalues have not been derived in \cite{MiRo}. What could be seen from the results of \cite{MiRo} is that this rate of convergence is different from the one for the case of a homogeneous material; in fact, they converge faster.

In the present paper we consider, in a more general setting, the question of the asymptotics of eigenvalues of a self-adjoint zero order pseudodifferential operator as they converge to an extremal value of an eigenvalue branch of the principal symbol.  We understand  that the behavior of these eigenvalues should depend on the structure of this extremal point. In this paper we  consider the case of a nondegenerate extremal value and a special case of a degenerate one, generalizing \eqref{princ}. The rate of convergence of eigenvalues depends, in fact, on the structure of the extremum. Some other cases of the structure of the extremal point will be considered in  later publications.

Our approach is based upon a reduction of the above spectral problem to the study of the negative eigenvalue  asymptotics of Schr\"odinger-like operators with negative potential tending to zero at infinity. This kind of problems has been considered since quite long ago, probably, starting from the papers \cite{BrCl}, \cite{R.Schr}, and \cite{Tamura77}, and further on, until very general results, in the pseudodifferential setting,  obtained by V.Ivrii, \cite{Ivrii0}, Sect.10.5. The formulas for eigenvalue asymptotics for this kind of problems have Weyl form, i.e., are expressed in the terms of phase space volume. However it is possible here that the region of the phase space where the symbol of the operator is negative, has finite volume, and in this case there are only finitely many negative eigenvalues, due to CLR-type estimates. This circumstance imposes the condition of a sufficiently slow decay of the absolute value of the potential, required for the validity of asymptotic formulas. As applied to pseudodifferential operators under consideration, this slow decay condition translates into the one of nonvanishing of the subprincipal symbol at the critical point.

\subsection{Setting 1. The pseudodifferential operator} \label{setting1}
Let $\Xb$ be a smooth compact boundary-less manifold of dimension $d$ (equipped with Riemannian metric.) Let $\Ec$ be an $\Nb-$
 dimensional Hermitian vector bundle over $\Xb.$ We consider a zero order pseudodifferential operator $\AF$ acting in
  the space of smooth sections of $\Ec$. We suppose that $\AF$ is selfadjoint in $L_2(\Ec)$ with respect to the Riemannian
   measure on $\Xb$ and the fixed Hermitian structure on $\Ec.$

In a fixed local co-ordinate system in a neighborhood $U$ in $\Xb$ and a fixed local frame in $\Ec$ over $U$, operator $\AF$ is defined by
\begin{equation}\label{Psdo}
    (\AF u)(x) =\Fc^{-1}_{\x\to x}\aF(x,\x)\Fc_{x'\to\x}u(x')+(\Oc u)(x)
\end{equation}
for a section $u$ of $\Ec$ supported in $U$. Here $\Fc$ is the Fourier transform, $\Fc^{-1}$ is the inverse Fourier transform, $\aF(x,\x)$ is the symbol of the operator $\AF$ in this local representation, a smooth section of  $\Hom(\Ec)$ and $\Oc$ is an infinitely smoothing operator.  The operator $\AF$ is supposed to be a classical zero order pseudodifferential operator; this means that in a fixed (and, therefore, any) local representation the symbol $\aF(x,\x)$ expands in an asymptotic series in homogeneous functions,
\begin{equation}\label{expan}
    \aF(x,\x)\sim \sum_{\n=0}^{\infty}\aF_{-\n}(x,\x), \, \aF_{-\n}(x,\t \x)=\t^{-\n}\aF_{-\n}(x, \x),\, \t>0.
\end{equation}
Of course, the symbol $\aF(x,\x)$ and its homogeneous components depend on the choice of local co-ordinates and the local frame. However the principal symbol $\aF_0(x,\x)$ and the subprincipal symbol $\aF_{\sub}(x,\x)=\aF_{-1}(x,\x)+\frac{1}{2i}\sum_j\partial_{\x_j}\partial_{x_j}\aF_0(x,\x)$ are invariant in the usual sense, see, e.g., \cite{Shubin}. (We, in fact, do not need to go into details of this invariance since our main considerations take place in  fixed local coordinates and a fixed frame.)

Consider the principal symbol $\aF_0(x,\x)$ of $\AF,$  for $(x,\x)\in \dot{T}\Xb$, the cotangent bundle of $\Xb$ with the zero section removed. It is a smooth function with values in $\Hom\Ec_x,$ zero order positively homogeneous in $\x$ variable. The self-adjointness condition implies  that $\aF_0(x,\x)$ is symmetric with respect to the Hermitian structure of the bundle $\Ec.$ Thus, $\aF_0(x,\x)$ has $\Nb$ real eigenvalues $\bm_\i(x,\x), \, \i=1,\dots,\Nb$ counting multiplicity, for any $(x,\x)\in \dot{T}\Xb.$ Ordered in the non-decreasing way,
$\bm_1(x,\x)\ge \bm_2(x,\x)\ge\dots\ge \bm_\Nb(x,\x),$ these eigenvalues are zero order homogeneous in $
\x$ and are continuous functions of $(x,\x)\in\dot{T}\Xb. $ Moreover, on a connected open  set $\Uc\subset  \dot{T}\Xb$ where a certain $\bm_\i(x,\x)$ has constant multiplicity, this eigenvalue is a smooth function of $(x,\x)\in\Uc.$ On such a set,  the corresponding eigenvectors $\eb_\i(x,\x)$ of $\aF_0(x,\x)$ can be locally chosen depending on $(x,\x)\in\Uc$ in a smooth way as well. This happens, in particular, if the eigenvalue $\bm_\i(x,\x)$ is simple.

More generally, let for $(x,\x)\in\Uc$ the eigenvalues $\bm_\i(x,\x)$ can be split into two disjoint groups, $\Ib_{\Nb}\equiv[1,\Nb]=\Ib\cup\Ib',$ so that $\bm_\i(x,\x)\ne \bm_{\i'}(x,\x), \, \i\in\Ib,\, \i'\in\Ib'.$ Denote by $\pi_\Ib(x,\x)$ the spectral projection of the homomorphism $\aF_0(x,\x)$ corresponding to the spectral set $\bm_\i, \i\in \Ib$, with similarly defined $\pi_{\Ib'}(x,\x)$. Under these conditions, the projections $\pi_\Ib(x,\x), $ $\pi_{\Ib'}(x,\x)$ depend smoothly on $(x,\x)\in \Uc.$ We are interested especially in the case when the set $\Ib$ consists of just one element.

\subsection{Setting 2. The  extremal point}
We give here the description of the structure of a tip point of the essential spectrum of  $\AF.$
We suppose that the principal symbol $\aF_0(x,\x)$ with eigenvalues $\bm_\i(x,\x)$ possesses the following properties. There exists a point $x^0\in \Xb$ such that $\bm_1(x^0,\x)=1$ for $\x\in \dot{T}_{x^0}\Xb,$ and  $\bm_1(x,\x)$ has a nondegenerate maximum at the point $x=x^0$.  Consider local co-ordinates on $\Xb$ near the point $x^0$ so that $x^0=0.$ Thus, near this point  the function
  $\bm_1(x,\x)$ has the form
\begin{equation}\label{m_1}
\bm_1(x,\x)=1 -\frac12\sum_{j,k}x_jx_k \zeta_{j,k}(\x) +O(|x|^3)\equiv 1-\Qb_{\x}(x)+O(|x|^3),
\end{equation}
where $\zeta_{j,k}(\x)=\partial_{x_j}\partial_{x_k}\bm_1(x,\x)|_{x=0}.$
By our nondegeneration assumption,  the quadratic form  $\Qb_{\x}(x)$ in \eqref{m_1} is positive definite for all $\x$ -- by compactness, it is uniformly positive definite for $\x\in \dot{T_{x^0}}\Xb$,
\begin{equation}
\sum_{j,k}x_jx_k \zeta_{j,k}(\x)\ge \g_0|x|^2, \, \g_0>0.
\end{equation}

We suppose further that $\bm_1(0,\x)=1$ is the global maximal value of $\bm_1(x,\x)$ for $(x,\x)\in \dot{T}\Xb$: there exists $\d>0, \rb>0,$ such that $\bm_1(x,\x)<1-\rb$ for all $x\in \Xb$ outside the $\d$ -neighborhood $U_{\d}$ of  $x_0$ and other eigenvalues $\bm_\i(x,\x),$ $\i>1,$ are strictly less than $1,$ $\bm_\i\le 1-\rb$ for all $x\in \Xb$, while $\bm_1(x,\x)<1$ for $x\in U_{\d}\setminus x^0$.

It follows that $1$ is the highest  point of the essential spectrum of $\AF.$  We are interested in finding the asymptotics of the   eigenvalues of $\AF$ as they approach $1,$ from above, of course.  We denote by $\l_j(\AF)$ these eigenvalues, $\l_1\ge \l_2\ge\dots>1$ and by $n(1+t)$ the counting function of these eigenvalues, $n(1+t)=\sum_{\l_j(\AF)>1+t}1.$

At some moment in our considerations, it will be convenient to pass from the operator $\AF$ to $\pmb{\AF}=1-\AF;$ for the new operator, it is the point zero will be the lowest tip of the essential spectrum, the smallest eigenvalue of the principal symbol $1-\aF_0$ will have a nondegenerate minimum at the point $x=0$ and the object of the study becomes the distribution of the negative eigenvalues of $1-\AF$ as they approach zero from below.

The most simple case is the scalar one.
Here $\aF(x,\x)$ is a scalar function and the above conditions are satisfied with $\bm_1(x,\x)$ replacing $\aF(x,\x)$

The first main theorem, concerning the scalar case, is the following.

\begin{thm}\label{Th.Scalar}Let the dimension of the bundle $\Ec$ equal $1$. Let the (only) eigenvalue of the principal symbol $\aF_0(x,\x)$ of the zero order pseudodifferential operator $\AF$ satisfy condition \eqref{m_1}. Then the eigenvalues of the operator $\AF$ converging to $1$ satisfy the asymptotic formula
\begin{equation}\label{Asympt.Scalar.Intro}
    n(\AF,1+t)\sim C_1(\AF)t^{-\frac{d}{2}},
\end{equation}
where the coefficient $C_1(\AF)$, see \eqref{Coef.1D.3}, is expressed in terms of the quadratic form $\Qb_{\x}(x)$ in \eqref{m_1} and the value subprincipal symbol of the operator $\AF$ at the point $x=x^0.$
\end{thm}

In the vector case the formulation of the theorem on eigenvalue asymptotics is more complicated, and it involves an additional condition of topological character (in fact, this condition is restrictive only in the dimension $d=3$).
The treatment of the general case, $\Nb>1$, will be based upon a partial diagonalization of the operator $\AF.$ This topic was recently under  investigation in the paper \cite{CRSV}. We present the result we need.

Consider the restriction of the vector bundle $\Ec$ to a neighborhood of the point $x^0=0$, where the condition \eqref{m_1} is satisfied for the eigenvalue $\bm_1(x,\x)$. We may suppose that this restriction is a trivial bundle. For the point in $\Xb,$ namely, for $x=x^0$, the  (one-dimensional) spectral subspace of the principal symbol $\aF_0(x^0,\x) $ corresponding to the eigenvalue $\bm_1(x^0,\x),$ $\x\in S^{d-1}=S^*_{x^0}\Xb,$ composes   a one-dimensional complex projective bundle $\Bc_1$ over the sphere $S^{d-1},$ which can be considered as a smooth mapping

\begin{equation}\label{bundle}
 \pi_1: S^{d-1}\to \operatorname{CP}^{\Nb-1},
\end{equation}
while the span of other eigenspaces of $\aF(\x)$ forms an $\Nb-1$- dimensional projective bundle $\Bc^{\bot}.$ We need to find a smooth global branch of the eigenvector $\eb_1(\x),$ $\x\in S^{d-1};$ this means, the section  of the projective bundle $\Ec_1$. There is a topological obstacle. Namely, such a branch exists if and only if one of the following conditions is satisfied:

\begin{itemize}\item \emph{Condition A}: $d\ne 3$; or
\item \emph{Condition B}: $d=3 $ and the Euler class $e(\Bc_1)\in H^2(\operatorname{CP}^{\Nb-1})$  pulls back to zero under the mapping \eqref{bundle}.
\end{itemize}

\begin{thm}\label{Thm.Vector} Suppose that $\AF$ is a zero pseudodifferential operator with symbol satisfying \eqref{m_1}. Suppose that \emph{Condition A} or \emph{Condition B} is satisfied. Then for the eigenvalues of $\AF$ converging to the point $1$ from above the asymptotic formula 
\begin{equation}\label{Asympt.Vector.Intro}
    n(\AF,1+t)\sim C(\AF)t^{-\frac{d}{2}},
\end{equation}
holds, where the coefficient $C(\AF)$ is expressed via the quadratic form $\Qb_{\x}(x)$ in \eqref{m_1}, the eigenvector corresponding to $\bm(x^0,\x),$ and the value of the subprincipal  symbol $\aF_{-1}$ at $x=x^0.$
\end{thm}

In the paper, we discuss in Sect. 2,3,4 various aspects concerning the reduction of our spectral problem to the one for a Schr\"odinger-type operators, as well as the reduction of a vector problem to the scalar one. In Sect. 5 we collect these considerations to give the proof of our main results. Finally, in Sect.6, we show how our results apply to the motivating
example of the Neumann-Poincar\'e operator in 3D-elasticity.

\subsection{Remarks} 
\begin{remark}The explicit expression for the coefficient in \eqref{Asympt.Vector.Intro} is rather unwieldy and not informative, therefore we do not present it here.
\end{remark}
\begin{remark}The results on the eigenvalue asymptotics can be extended to other forms of the behavior of eigenvalue $\bm_1(x,\x)$ near the critical point. For example, if the leading terms of the Taylor expansion of the eigenvalue $\bm_1(x,\x)$ at $x=0$ has the form 
\begin{equation}\label{general eigenvalue}
    \bm_1(x,\x)=1-Q(x,\x)+o(|x|^{2l}),
\end{equation}
where $Q(x,\x)$ is a homogeneous polynomial in $x$ of degree $2l$ such that $Q(x,\x)\ge C|x|^{2l}$, the asymptotics of eigenvalues is
\begin{equation}\label{General.asymp}
    n(1+t, \AF)\sim C^{(2l)}(\AF) t^{-\frac{d}{2l}}.
\end{equation}
Such results are considerably more laborious; they are based upon a more advanced machinery of V.Ivrii dealing with  the eigenvalue asymptotics for Schr\"o\-dinger type operators. In this direction, the eigenvalue asymptotics can be found also for operators whose principal symbol has at the critical point a non-isotropic maximum, say, $\bm_1(x,\x)\sim 1-x_1^2-x_2^4,$ $(x_1,x_2)\to 0,$ in dimension $d=2.$ The corresponding results will be published elsewhere.
\end{remark}

\section{Extraction of the scalar operator} 
In considering operators acting on vector functions, we reduce the spectral problem to the one for the scalar operator. We will need this in the special case of operators on the Euclidean space (thus certain topological obstacles arising in a more general case disappear.) We use the constructions elaborated in \cite{Capo.Diag}, \cite{CRSV}. 

Let $\AF$ be a zero order pseudodifferential operator in $\R^d$ acting on the vector functions of dimension $\Nb$. We suppose that its principal symbol, the Hermitian matrix $\aF_0(x,\x)$ stabilizes at infinity:
\begin{equation}\label{stabil}
    \aF_0(r\s,\x)\to \aF_0^{\infty}(\s,\x), \, r\to+\infty, \s=\frac{x}{r}, \, r=|x|,
\end{equation}
moreover, $\aF_0(x,\x)=\aF_0^{\infty}(\s,\x)$ for sufficiently large $r=|x|.$
We suppose further on that the spectrum of the matrix $\aF_0(x,\x)$ is split in the following way:
there exist disjoint closed intervals $\Jc_1,\Jc'\subset \R^1$ such that one eigenvalue $\bm_1(x,\x)$ of the matrix $\aF_0(x,\x)$ lies always in $\Jc_1,$ and the remaining spectrum of the matrix $\aF_0(x,\x)$  lies in $\Jc'.$ Let the \emph{Condition A} or \emph{Condition B} above be satisfied. Then it is possible to find a unitary pseudodifferential operator $\TF=\Tc(x,D)$ such that $\TF^*\AF\TF$ diagonalizes $\AF$ up to lower order terms, of order $-2$ in our case. Namely, the space $L_2(\Xb,\Ec)$ splits into the orthogonal direct sum
 $$L_2(\Xb,\Ec)=L_2(\Xb,\Ec_1)\oplus \Hs^\bot$$
with a one-dimensional trivial bundle $\Ec_1,$
such that
\begin{equation}\label{Diag.Cap}
  \TF^*\AF\TF = \diag(\AF_{(1)}, \AF_{\bot})  +\RF,
\end{equation}
where $\RF$ is an operator of order $-2.$ with symbol decaying as $|x|^{-2} $ at infinity. So, a scalar pseudodifferential operator, responsible for the spectrum near the extremal point of $\bm_1(x,\x),$ is separated.  

We explain here how the above  transformation is constructed. Let $\G_1,\G' $ be non-intersecting smooth contours in $\C^1$ encircling, respectively, intervals  $\Jc_1,\Jc'.$ Since the spectrum of $\AF$ outside $\Jc_1\cup \Jc'$ is finite, these contours can be chosen in such way that they do not pass through eigenvalues, and contain all eigenvalues inside. Denote $\Si_1, \Si
'$ the spectrum of $\AF,$ correspondingly, inside  $\G_1,\G' $ and by $\Pi_1,\Pi'$ the corresponding spectral projections of $\AF.$
Using the spectral theorem, we split $\AF$ into the direct sum, $\AF=\AF_1\oplus\AF',$ where $\AF_1=\Pi_1, \AF'=\Pi'\AF.$ The Hilbert space $\Ht=L_2(\Ec)$ splits accordingly into the direct sum $L_2=\Ht_1\oplus\Ht'\equiv\Pi_1 \Ht\oplus\Pi_1 \Ht.$

Operators $\AF_1,$ $\AF'$ are pseudodifferential operators of order zero. This follows from the F.Riesz representation of the spectral projectors: $\Pi_1=(2\pi i)^{-1}\int_{\G_1}(\AF-\z)^{-1} d\z.$ The resolvent, $(\AF-\z)^{-1}$ is a zero order pseudodifferential operator. Its principal symbol is $(\aF_0(x,\x)-\z)^{-1}$ and therefore the principal symbol of $\AF_1$ is
$(2\pi i)^{-1}\int_{\G_1}(\aF_0(x,\x)-\z)^{-1} d\z,$ which equals $\aF_0(x,\x)$ times the spectral projection $\pi_1(x,\x)$ of the Hermitian matrix $\aF_0(x,\x)$ corresponding to its spectrum inside the contour $\G_1.$

We are interested in the case when  the interval $\Jc_1$ contains only one eigenvalue of the symbol $\aF_0(x,\x),$ namely, $\bm_1(x,\x).$ This means that the principal symbol of $\AF_1$ is the rank one operator $\bm_1(x,\x)\pi_1(x,\x).$

Lower order symbols of $\AF_1$ can be calculated iteratively using the Neumann series for the resolvent $(\AF-\z)^{-1},$ we however do not perform all these calculations since the explicit formulas for these terms are of no interest in the moment. We explain the calculation of the symbol of order $-1$ only.

Namely, let $\bF$ be the order $-1$ symbol of $\AF$ (in the same fixed local co-ordinate system and frame as $\aF_0.$
We are interested in  the first two terms of the symbol of the resolvent as $\rF[\z]=(\aF-\z)^{-1} +\cF,$ and thus the symbol $\cF$ we are looking for should satisfy the equation
 \begin{equation}\label{resolv1}
    (\aF_0-\z +\bF)\circ ((\aF_0-\z)^{-1}+\cF) =1+\mF,
 \end{equation}
where $\mF$ is a symbol of order $-2.$
 Equation \eqref{resolv1} gives us
\begin{equation*}
    (\aF_0-\z)\cF +\bF(\aF_0-\z)^{-1}-\frac{1}{2i}(\partial_x \aF_0)(\aF_0-z)^{-1}(\partial_\x \aF_0) (\aF_0-z)^{-1}=0,
\end{equation*}
therefore
\begin{equation}\label{resolv2}
    \cF(\z)= -(\aF_0-\z)^{-1}\qF(\z)(\aF_0-\z)^{-1},
\end{equation}
where
\begin{equation}\label{resolv3}
    \qF(\z)=\bF+\frac{1}{2i}(\partial_x \aF_0)(\aF_0-z)^{-1}(\partial_\x \aF_0).
\end{equation}
Finally, the order $-1$ symbol of $\AF_1$ equals
\begin{equation}\label{lower.symbol}
    \aF_{1,-1}=\frac{1}{2\pi i}\int_{\G_1}\cF(\z) d\z,
\end{equation}
where $\cF(\z)$ is given in \eqref{resolv2}, \eqref{resolv3}.

So, we have split, up to a nonessential error, our pseudodifferential operator into the direct sum of pseudodifferential operators operators. For further construction, we note that
for the pseudodifferential operator $\AF_1$, the range of the principal symbol $\aF^{1}_0(x,\x),$ forms a linear complex   bundle over $S^*\R^d,$ a subbundle in the bundle $\qb^*\Ec,$ where $\qb$ is the projection $\qb:S^*\R^d\to\R^d.$ Assuming that Condition  A or  Condition B is satisfied, $\qb^*\Ec$
admits a global section $\vb(x,\x).$ Having such global section, the construction in   \cite{Capo.Diag} produces an isometric pseudodifferential  operator $\TF$ from  $L_2(\R^d)$ onto  $\Ht_1$ such that $\TF^*\AF_1\TF$ is, up to an infinitely smoothing operator, a scalar pseudodifferential pseudodifferential operator. We, in fact, do not need an infinitely smoothing error, it is sufficient to have an error operator of order $-2.$

We can summarize the above construction as the following proposition.
\begin{Proposition}\label{Prop.Extract}
Suppose that for a symbol $\aF_0(x,\x)$  the eigenvalues $\bm_1$ and $\bm_\i,$ $\i>1,$ belong, respectively, to disjoint closed intervals $\Jc_1,$ $\Jc'$, together containing all spectrum of $\AF$ Let $\AF_1, \AF'$ be the pseudodifferential operators, parts of $\AF$ corresponding to the spectrum in $\AF_1, \AF'.$ Suppose that  one of conditions A, B is satisfied. Then there exists an isometric pseudodifferential operator $\TF$ which implements, up to an operator of order $-2,$ a unitary equivalence between $\AF_1$ and a scalar pseudodifferential operator with principal symbol $\aF_1.$
 \end{Proposition}

Note that in \cite{Capo.Diag}, a procedure of finding the required transformation is presented in a technically different way, which gives an explicit expression for the lower order symbol of $\AF_1$ not using contour integration. 

\section{Eigenvalue asymptotics for Schr\"odinger-like operators}\subsection{Some history}Schr\"odinger-like operators are operators of the form $\Hb=\Ab-\Bb,$ where $\Ab$ is a positive  order $l$ elliptic (pseudo-)differential  operator in $\R^d$ with symbol $\ab(x,\x)\asymp|\x|^{2l}$  and $\Bb$ is a zero  order operator with symbol decaying as $x$ tends to infinity. The initial class of such operators, studied at least since early XX century, is the second order ($l=1$) Schr\"odinger operator with Coulomb type potential, $\Hb=-\D-V(x)$ in $\R^d$, with $V(x)\sim q|x|^{-\g}$ as $|x|\to\infty$ with $\g>0,$ $q>0.$ Spectral properties of of this operator can be studied in a standard way, using proper special functions and elementary variational and perturbation considerations. The essential spectrum coincides  always with the positive semi-axis. As for the discrete spectrum, consisting of negative eigenvalues, everything depends on the rate of decay of $V(x)$ at infinity. In particular, if $\g>2,$ i.e., the 'potential' $q|x|^{-\g}$ decays fast at infinity, the operator $\Hb$ has finite discrete spectrum, and the number of negative eigenvalues is controlled, e.g.,  by the CLR estimate (for $d\ge 3$) or proper eigenvalue estimates in low dimensions. On the other hand, if $\g\in(0,2),$ there are infinitely many negative eigenvalues, and their distribution, described by the function $n(-t;\Hb):=\#\{j: \l_j<-t\}$ is a separate problem. The typical result for the Schr\"odinger operator is the asymptotic phase volume formula
\begin{equation}\label{As.Neg.Schr}
    n(-t,\Hb)\sim (2\pi)^{-d} \vol_{\R^d\times\R^d}\{(x,\x):\Hc(x,\x)<-t\},\, t\to 0,
\end{equation}
where $\Hc(x,\x)$ is the classical Hamiltonian, $\Hc(x,\x)=|\x|^2-V(x).$ This kind of formulas is a more delicate fact than the semi-classical ones, say,
\begin{gather*}
n(-t, \Hb_h)\sim (2\pi)^{-d}h^{-{d}}\vol\{(x,\x):\Hc(x,\x)<-t\},\,h\to 0,\\\nonumber\Hb_h=-h^{2}\D-V(x),\,\Hc_h(x,\x)= |\x|^2-V(x),
\end{gather*}
where the phase volume  of a fixed domain in the phase space is present, while in \eqref{As.Neg.Schr} the domain changes depending on the spectral count parameter $t.$ Formula  \eqref{As.Neg.Schr} has been sequentially extended to more and more general classes of potential $V(y), $ see \cite{BrCl}, \cite{R.Schr},  \cite{Tamura.rem}, \cite{Tamura77}, as well as to operators in  a more general setting in \cite{Levendorsky}, \cite{LevenBook}, \cite{Barbe}, \cite{Ivrii0}, including the matrix case.

 We give here an exposition of the results of \cite{Levendorsky}, Sect.9,  see also a more detailed presentation in \cite{LevenBook}, Sect.21,   for the particular case of operators of interest that we study in this paper. More general Schr\"odinger-like operators, needed for the treatment of more complicated cases of zero order pseudodifferential operators are presented in \cite{Ivrii0}, \cite{Ivrii}.

 The operators in question act on vector-functions on $\R^d$ with values in $\C^{\Nb},$ so, all symbols and their components are matrix-valued functions on $\R^d\times\R^d$ with values in $\End(\C^{\Nb}).$

 Pseudodifferential operators in \cite{Levendorsky} are defined by means of the Weyl quantization;  this means that with the symbol $a_W(x,\x)$ in a proper class one associates the operator
\begin{equation}\label{Weyl}
    OP_W(a_W)u (x)=(2\pi)^{-d}\int_{\R^d}e^{ix\x}\int_{\R^d}e^{-iy\x}a_W(x+y,\x)u(y)dyd\x.
\end{equation}
As usual, formula \eqref{Weyl} defines the operator initially on functions on the Schwartz space and then extends by continuity to the proper Sobolev space (depending on the quality of the symbol $a_W$.) We recall here that other pseudodifferential quantizations are used, we will, in particular, apply the  quantization
\begin{equation}\label{Left}
     OP_\ell(a_\ell)u (x)=(2\pi)^{-d}\int_{\R^d}e^{ix\x}\int_{\R^d}e^{-iy\x}a_\ell(x,\x)u(y)dyd\x.
\end{equation}
This quantization is called 'left' in \cite{Levendorsky}, \cite{LevenBook} (therefore the subsymbol $\ell$), while it is called $0$-quantization, in the scale of $\t$-quantizations, $\t=0$, in the classical book by M.Shubin, \cite{Shubin}; in the terms of this latter book, the Weyl quantization is the $\frac12$-one. There exists  the universal relation connecting symbols of the given operator presented in different quantizations, see Theorem 23.3 in \cite{Shubin};  in our case, for symbols $a_W, a_\ell$ such that the operator $OP_W(a_W)-OP_\ell(a_\ell)$ is negligible, the relation holds
\begin{gather}\label{W-l relation}
    a_W(x,\x)\sim \sum_\a\frac{1}{\a!}\left(-\frac12\right)^{|\a|}\partial_\x^\a D_x^\a a_\ell(x,\x),\\\nonumber
    a_\ell(x,\x)\sim \sum_\a\frac{1}{\a!}\left(\frac12\right)^{|\a|}\partial_\x^\a D_x^\a a_W(x,\x), \, D_x=-i\partial_x.
\end{gather}

The conditions in \cite{Levendorsky} are imposed of the Weyl symbol of the operator. We will see later what form do these conditions take for the $\ell$-symbol.
We consider a special form of the  symbol $a_W(x,\x)$ fitting in  Theorem 9.1 in \cite{Levendorsky}. Since we are not interested in remainder estimates in eigenvalue asymptotic formulas, we skip some conditions in the  theorem needed for these remainder estimates only.

Our symbol $a_W(x,\x),$ up to weaker terms which are denoted by $a_1(x,\x)$ in (9.6) in \cite{Levendorsky}  and do not influence the leading terms in the eigenvalue asymptotics, equals  $\tilde{a}(x,\x)$
\begin{gather}\label{SymbolLev}
    \tilde{a}(x,\x)=\sum_{|\a|=2}a_{\a,0}(\o)\x^{\a}+\\\nonumber
    +\sum_{|\a|=1}a_{\a,1}(\o)\x^\a|x|^{-\k}-a_2(\o)|x|^{-2\k}, \, \o=\frac{x}{|x|},
\end{gather}
for some $\k\in(0,1).$
The ellipticity condition (9.7) in \cite{Levendorsky} requires that, for the leading term, $|\a|=2, s=0,$
\begin{equation}\label{ellipt}
    \sum_\a a_{\a,0}(\o)\x^\a\ge c|\x|^2.
\end{equation}
Finally, there is the  positivity condition (9.5): there exist a nonnegative $l_0$ such that for any $\e>0$ there exist positive constants $c(\e), C(\e)$ such that \begin{equation}\label{Posit.Lev}
\tilde{a}(x,\x)\ge c(\e)|\x|^{l_0}
\end{equation}
for all $\x: |\x|<\e, |x|>C(\e).$
This condition for the symbol \eqref{SymbolLev} is satisfied for $l_0=l=2.$

Under these conditions, according to Theorem 9.1 in \cite{Levendorsky}, for the eigenvalues of the operator $\Hb=OP_W(a_W)$ the following eigenvalue asymptotics holds
\begin{equation}\label{Asymp.Lev}
    n(-t,\Hb)\equiv\#\{\l_j(\Hb)<-t\}\sim \Cb(a_W)t^{-\theta},
\end{equation}
where $\theta= \frac{d(1-k)}{kl}, $ which for our case, $k=\frac12,$ $l=2,$ equals $\theta=\frac{d}{2}.$
The asymptotic coefficient $\Cb(a_W)$ equals
\begin{equation}\label{Coeff.Lev}
    \Cb(a_W)=\sum_\i \meas\{(x,\x)\in \R^{d}\times\R^d: \bn_\i(x,\x)+1<0\},
\end{equation}
where $\bn_\i(x,\x)$ are the eigenvalues of the matrix $\tilde{a}(x,\x).$

Under the conditions of this Theorem, in the expression \eqref{Coeff.Lev}, the value of the coefficient $\Cb(a_W)$ is determined by the region in the phase space where $|\x|$ is small, so that the whole  symbol $\tilde{a}(x,\x)$ has negative eigenvalues.

We make here a detailed calculation for the
 case of our special interest, the scalar one, $\Nb=1,$ when $a_{\a,1}=0,$
Here, the eigenvalue $\n(x,\x)$ equals
\begin{equation}\label{ScalarCase}
    \n(x,\x)=\sum_{|\a|=2}a_{\a,0}(\o)\x^\a-a_2(\o)|x|^{-1}, \, \o=x/|x|\in S^{d-1}.
\end{equation}
Thus the coefficient $\Cb(a_W)$ takes the form
\begin{equation}\label{Coef.1D.1}
    \Cb(a_W)=\meas\{(x,\x): a_2(\o,\x)<a_{0,+}(\o)r^{-1}-1\}, \, r=|x|,
\end{equation}
where $a_2(\o,\x)=\sum_{|\a|=2}a_{\a,2}(\o)\x^\a.$ In the expression in \eqref{Coef.1D.1}, $a_{0,+}(\o)$ denotes the positive part of $a_{0}(\o)$ and we take into account that the points where $a_2(\o)$ is negative do not contribute to the right-hand side.  When calculating   the coefficient in \eqref{Coef.1D.1},
for fixed $\o,|x|$, we have
\begin{equation}\label{Coef.1D.2}
    \meas\{\x: a_2(\o,\x)<a_0(\o)r^{-1}-1<0\}=\Vb_d\det[a_2(\o)]^{-\frac{1}{2}}(a_{0,+}(\o)r^{-1}-1)_+^{\frac{d}2},
\end{equation}
where $\pmb{\Omega}_d$ is the volume of the unit ball in $\R^d.$
Integration of the expression in \eqref{Coef.1D.2} over $x$ gives now
\begin{gather}\label{Coef.1D.3}
    \Cb(a_W)=\pmb{\Omega}_d\int_{S^{d-1}}(\det[a_2(\o)])^{-\frac{1}{2}}\int_0^\infty(a_{0,+}(\o)r^{-1}-1)_+^{\frac{d}{2}}r^{d-1}drd\o\\\nonumber
=\pmb{\Omega}_d \mathbf{B}(\frac{d}2+1, \frac{d}2)\int_{S^{d-1}}(\det[a_2(\o)])^{-\frac{1}{2}}a_{0,+}(\o)^{\frac{d}{2}}d\o,
\end{gather}
where $a_2(\o)$ is the matrix $[a_{\a,2}(\o)]_{|\a|=2},$
 and $\Bb$ is the Euler Beta-function.

Thus, we arrived at our general eigenvalue asymptotics theorem that we will use for the study of the discrete  spectrum of zero order operators.
\begin{thm}\label{Gen.Thm.} Let $\Hb$ be the operator in $L_2(\R^d) $ with Weyl symbol $a_w=a_2(x,\x)-h(x);$  let $a_2(x,\x),$ the second order term in the symbol of the operator $\Hb$ be a positive quadratic form in $\x$ with coefficients depending on $x,$ zero order positively homogeneous in $x$ with smoothing near $x=0$. Suppose that $h(x)$ decays as $a_0(\o)|x|^{-1},$ $\o=\frac{x}{|x|}\in S^{d-1},$  as $|x|\to\infty.$
Then for the operator $\Hb = a_2(x,D)-h(x)$ the negative eigenvalues satisfy the asymptotic law
\begin{equation}\label{As.Gen.Thm.}
    n(-t,a_2(x,D)-h(x))\sim t^{-\frac{d}{2}}\Cb(a_W),
\end{equation}
where $\Cb(a_W)$ is defined in \eqref{Coef.1D.3}.
\end{thm}

\section{Localization and perturbations}\label{Sect.Transf}
In this section we justify the transformations of the initial spectral problem for the zero order pseudodifferential operator, having in mind to reduce it finally to a Schr\"odinger type operator in the next section. We note first that, as explained above, the essential spectrum of a zero order pseudodifferential operator consists of one or several intervals (which may degenerate to single points). The study of the behavior of eigenvalues converging to a certain  tip of the essential spectrum, i.e., to an endpoint of one of these intervals, can be reduced, by a simple linear-fractional transformation of the operator, to the same problem for the lowest -- or for the highest, if needed, -- point of the essential spectrum; we will use such transformation without additional comments.

\subsection{Localization}

We suppose that the symbol of our operator $\AF$ has the structure described in Sect.1. 

Consider a smooth cut-off function $\h(x), \, x\in \Xb$ which equals $1$ near the point $x^0,$ in  $U_\d$, and vanishes outside another,  $2\d$- neighborhood of this point. Set $\h'=1-\h.$ Then the pseudodifferential operator $\AF$ splits into the sum
\begin{equation}\label{splitting}
    \AF=\BF+\CF, \, \BF=\h\AF\h,\, \CF=\AF-\BF.
\end{equation}
Operator $\CF$ is a zero order self-adjoint  pseudodifferential operator with principal symbol $\cF_0(x,\x)=(1-\h(x)^2)\aF(x,\x).  $  The eigenvalues of this principal symbol equal $(1-\h(x)^2)\bm_\i(x,\x).$ Therefore the essential spectrum of $\CF$ lies in the range of the functions $(1-\h(x)^2)\bm_\i(x,\x)$, and consequently below $1-\rb$. The spectrum of $\CF$ above $1-\rb$ is therefore discrete. So, there are only a finite number of eigenvalues of $\CF$ above $1$. Therefore, the asymptotics of the eigenvalues of $\AF$ approaching $1$ from above is the same as the one for the operator $\BF$, provided there are infinitely many of the latter ones.
 Such 'localization' property can also be formulates in another, more convenient way.
\begin{lem}\label{localization}Let $\AF$ be a self-adjoint zero order pseudodifferential operator on $\Xb$ such that the largest eigenvalue $\bm_1(x,\x)$ attains it maximal value $1$ at the point $x^0$ for all $\x$ and this maximum is nondegenerate, while all other eigenvalues $\bm_\i(x,\x)$ are always less than $1-\rb$, $\rb>0$. Then the asymptotics of eigenvalues of $\AF$ approaching the point $1$ does not depend on the values of the symbol outside any neighborhood of $x_0.$
\end{lem}

Having this localization in mind, we can suppose that the operator $\AF$ contains the cut-off function $\h(x)$ from the very beginning, and thus acts only in a neighborhood of $x^0.$ Therefore, we can consider our operator as acting in a (small) domain in the Euclidean space $\R^d$, having symbol with compact support, or on a sphere, with symbol supported in a small cap around the pole. Note that until now we have not changed the symbol near the point $x^0.$

Next, for the sake of convenience of reasoning, we pass to considering, instead of the above $\AF$, the operator $\pmb{\AF}=1-\AF,$ in $\R^d.$ For the operator $\pmb{\AF},$ the smallest eigenvalue of the principal symbol $\pmb{\aF}=1-\aF$ has minimal value at $x=0,$ while the complete symbol equals $1$ outside a neighborhood of $0$.  Thus, zero is the \emph{lowest} point of the essential spectrum and our problem is reduced to studying the asymptotics of the negative eigenvalues of $\pmb{\AF}.$ From now on, we replace the notation $\pmb{\AF}$ by the old one, $\AF,$ with corresponding notation change for the symbol, its components and its eigenvalues. So, the smallest eigenvalue $\bm_1(x,\x)$ has now  a nondegenerate minimum at $x=0,$ $\bm_1(x,\x)=Q(x,\x)+O(|x|^3), \, x\to 0,$ and all other eigenvalues are larger than certain $\rb>0.$

\subsection{Perturbations}\label{Perturbations}Here we present some results on the behavior of the counting function for the eigenvalues under perturbations which are in some sense weak. Mostly, they are variations of well-known properties but adapted to our situation.

\begin{lem}\label{Lem.perturb}Let $\AF$ be a non-negative self-adjoint operator. For an operator $\VF$ we denote by $n(-t, \AF-\VF)$ the number of eigenvalues of $\AF-\VF$ below $-t.$ Suppose that $\WF$ is a weak perturbation in the sense that for a certain $\g>0,$ for any $\ve>0,$

\begin{equation}\label{weak}
n(-t, \ve\AF-\WF)=o(t^{-\g}), t\to 0,
\end{equation}
then
\begin{gather}\label{pert.estimates}
    \limsup_{t\to 0}t^\g n(-t,\AF-\VF-\WF)\le\limsup_{t\to 0}t^\g n(-t,\AF-\VF),\\\nonumber
\liminf_{t\to 0}t^\g n(-t,\AF-\VF-\WF)\le\liminf_{t\to 0}t^\g n(-t,\AF-\VF).
\end{gather}
If, additionally to \eqref{weak}, a similar relation holds with $-\WF$ instead of $\WF$, one can replace $'\le'$ by $'='$ in \eqref{pert.estimates}
\end{lem}
\begin{proof}We prove only the inequality
\begin{equation*}
\limsup_{t\to 0}t^\g n(-t,\AF-\VF-\WF)\le\limsup_{t\to 0}t^\g n(-t,\AF-\VF);
\end{equation*}
all remaining relations are proved analogously. As it follows from the variational principle, for a fixed $\ve>0,$
\begin{equation}\label{pert.ineq}
    n(-t,\AF-\VF-\WF) \le n(-t\ve, \ve\AF-\WF)+n(-t(1-\ve), (1-\ve)\AF-\VF.)
\end{equation}
We multiply \eqref{pert.ineq} by $t^\g$ and pass to $\limsup$ as $t\to 0$. The first term on the right vanishes, and by further letting $\ve\to 0$, we obtain the required inequality.
\end{proof}
 Another property concerns the monotonicity of the asymptotics of the counting function with respect to the main operator $\AF$ operator $\AF.$
\begin{lem}\label{Lem.monoton}Let $\AF,$ ${\BF}$ be two  pseudodifferential operators in $\R^d$ with nonnegative Weyl symbols $\aF,$ ${\bF}$ such that they are larger than $\rb>0$ for $x$ outside a $\d-$ neighborhood of $x=0$ and $\aF(x,\x)\le  {\bF}(x,\x)$ in these neighborhood.  Then for Weyl pseudodifferential operators $\AF,$ $\BF$,
\begin{gather}\label{ineq}
   \limsup_{t\to 0}t^\g n(-t,{\BF}-\VF)\le\limsup_{t\to 0}t^\g n(-t,\AF-\VF),\\\nonumber
\liminf_{t\to 0}t^\g n(-t,{\BF}-\VF)\le\liminf_{t\to 0}t^\g n(-t,\AF-\VF).
\end{gather}
\end{lem}
\begin{proof}First, we can change the symbols of operators $\AF,{\BF}$ outside a small neighborhood of zero such that the inequality  between symbols holds everywhere in $\R^d$. This transformation, by Lemma \ref{localization} does not change the asymptotics of negative eigenvalues. After this, \eqref{ineq} follows from the variational principle, since operators with positive Weyl symbol are positive.
\end{proof}

\subsection{Freezing the subsymbol}\label{Sect.lower order} In this section we discuss how the eigenvalue problem for a general zero order pseudodifferential operator can be reduced to a problem of a special form.


We recall that the asymptotic distribution of eigenvalues of $\AF$ above $1$ does not depend on values of symbol of the operator outside an arbitrary small neighborhood of $x^0.$ Using this fundamental property, we perform a series of transformation of our problem. In all these transformations we do not change the symbol $\aF$ for $x\in\Uc,$ therefore the asymptotics of eigenvalues the eigenvalues approaching the tip of the essential spectrum does not change.


Here we have fixed a global Euclidean co-ordinates system in $\R^d.$  Since now, all $\x$-homogeneous terms in the polyhomogeneous symbol of $\aF_1(x,\x)$ and further operators under consideration, are defined globally.

In these co-ordinates, we consider the subsymbol $\bF(x,\x)$. It is smooth in $x$ and, outside a neighborhood of $\x=0$ positively homogeneous of order $-1$ in $\x$.  We freeze the symbol at the point $x=0$ and consider the splitting $\bF(x,\x)=\bF(x_0,\x)+\cF(x,\x),$ where $\cF(x_0,\x)=0.$ Thus, $\cF(x,\x)=|x-x_0|\dF(x_0,\x)$ with a bounded function $\dF$.

The following proposition provides us with the possibility to freeze the subsymbol of the operator at the point $0,$ not changing the asymptotics of negative eigenvalues.
\begin{proposition}\label{Prop.Est}Let $\AF$ be a zero order pseudodifferential operator with Weyl symbol $\aF_W(x,\x)\ge p|x|^2, \, p>0$ for $|x|<\d$ and $\aF_W(x,\x)\ge 1,$ $|x|>\d$. Let $\CF$ be an order $-1$ operator with principal symbol  $\cF(x,\x)$, which is a symbol of order $-1,$ moreover, $\cF(0,\x)=0$.
Then for the eigenvalues of operator $\AF(x,D)-\CF(x,D)\equiv \AF-\CF$ the estimate holds
\begin{equation}\label{sub.est.1}
    n(-t, \AF(x,D)-\CF(x,D))=o(t^{-\g}), \, \mbox{\emph{for\, any}}\,    \g >0.
\end{equation}
\end{proposition}
\begin{proof} The proof follows the structure of the one of Theorem 4.3 in \cite{MiRo}. Let $\mF(x)$ be a smooth nonnegative function having zero of order 2 at $0$ and such that $\aF(x,\x)\ge \MF(x)$, where $\MF(x)$ is $\mF(x)$ times the unit matrix, $\mF(x)\to 1$ at infinity.  Therefore, $n(-t, \AF(x,D)-\CF(x,D))\le n(-t,\MF(x)-\CF ).$ By the variational principle, we are interested in the study of the quantity
\begin{equation}\label{sub.est.2}
    n(-t,\MF(x)-\CF )=\max\dim\{\Lc: ((\MF(x)-\CF)u,u)<-t\|u\|^2 , \, u\in \Lc\setminus\{0\}\}
\end{equation}
as $t\to +0$ (here and further on, the subspaces are considered in $L_2(\R^d)$.)

We write the inequality in \eqref{sub.est.2} as
\begin{equation*}
    \int \mF(x)|u|^2 dx- (\CF u,u)< t\int |u(x)|^2 dx,
\end{equation*}
or
\begin{equation}\label{sub.est.3}
    (\CF u,u)\ge \int (t+\mF(x))|u(x)|^2 dx.
\end{equation}
In the classical inequality $AB\le \frac{A^p}{p}+\frac{B^q}{q}$
for positive $A,B$ and $p^{-1}+q^{-1}=1$, we set $A=t^{\frac1p},B=|x|^{\frac2q}$, with $p,q$ to be fixed later.
Thus, we have:
\begin{equation}\label{sub.est.4}
    \int(t+\mF(x))|u(x)|^2dx \ge c t^{\frac1p}\int|x|^{\frac2q}|u(x)|^2 dx
\end{equation}
If we replace the right-hand side in \eqref{sub.est.3} by a smaller quantity, namely,
 by the right-hand side in \eqref{sub.est.4}, then the maximal dimension of subspaces where the resulting inequality
holds can only increase. Therefore, \eqref{sub.est.2}, \eqref{sub.est.3}, \eqref{sub.est.4} imply an upper estimate for the counting functions of eigenvalues we are interested in:
\begin{equation}\label{sub.est.5}
    n(-t,\MF(x)-\CF )\le \max\dim\{\Lc: (\CF u,u)>t^{\frac1p}\int|x|^{\frac1q}|u(x)|^2dx, \, u\in\Lc\setminus\{0\}\}.
\end{equation}
Next, denote $|x|^{\frac1q}u$ as $v$ in \eqref{sub.est.5}; this gives

 \begin{gather}\label{sub.est.6}
 n(-t,\MF(x)-\CF )\le \\\nonumber \max\dim\{\Lc:(|x|^{-\frac1q}(\CF(|x|^{-\frac1q}v),v)> ct^{\frac1p}\int|v(x)|^2dx,\, v\in\Lc\setminus\{0\}\}.
 \end{gather}
The quantity on the right-hand side in \eqref{sub.est.6} is nothing but the singular numbers distribution function $n(t^{\frac1p},\Zb)$, where $\Zb$ is the operator  $\Zb=|x|^{-\frac1q}\CF|x|^{-\frac1q}.$ Recall that $\CF$ is the order $-1$ pseudodifferential operator with principal symbol $\cF(x,\x)$ which vanishes for $x=0$. Therefore, up to terms of lower order (which can be easily shown to to make a weaker contribution to the eigenvalues estimates),
$\CF$ is an integral operator with kernel  majorated by $C|x||x-y|^{1-d}.$ In this way, we are left with studying  singular number   estimates for the integral operator with kernel $\kF(x,y)$ satisfying
\begin{equation}\label{sub.est.7}
    |\kF(x,y)|\le C |x|^{1-\frac1q}|x-y|^{1-d}|x|^{-\frac1q}.
\end{equation}
Such operators fit into the general setting of the fundamental paper \cite{BS.Int}, namely Theorem 10.3 there.

We cite here the conditions of this theorem in \cite{BS.Int} as applied to our case.
Let the function $F(x),$ $x\in\O\subset\R^d$ be positively homogeneous of order $\kb=1-d,$ (the condition $-d<\kb<0$ is fulfilled) and continuous in the angle variable. Let the weight functions $a(x), b(x)$ belong to $L_{\l_1}, L_{\l_2},$
where
\begin{equation}\label{sub.est.8}
    \l_1^{-1}+\l_2^{-1}=\d^{-1}\equiv 1+\frac{\kb}{d}=d^{-1};
\end{equation}
supposing $\d^{-1}\le \l_j\le\infty.$
Then, by Theorem 10.3, (a), for the integral operator $\Fb_{a,b}$ with kernel $a(x)F(x-y)b(y),$
\begin{equation}\label{sub.est.9}
    s_n(\Fb_{a,b})\le C(F) n^{-\frac1\d}\|a\|_{L_{\l_1}}\|b\|_{\l_2},
\end{equation}

Being applied to our case, with $a(x)=C|x|^{1-\frac1q}\in L_\infty$, $b(x)= C|x|^{-\frac1q},$
where $\l_1=\infty$ and $\l_2=d$, the conditions of theorem are fulfilled when $|x|^{-\frac1q}\in L_{d}$ which holds for any $q>1.$ With  such $q$ fixed, we have estimate \eqref{sub.est.9} which can be rewritten as
$n(s, \Fb_{a,b})=O(s^{-\frac1d})$. We set $s=t^{\frac1p}$, which, by \eqref{sub.est.6}, gives the required estimate
\begin{equation}\label{sub.est.10}
     n(-t,\MF(x)-\CF )=O(t^{-\frac{1}{pd}}),
\end{equation}
and this, by the arbitrariness of $p<\infty$, gives the required estimate $n(-t,\MF(x)-\CF )=o(t^{-\vp})$ for an arbitrary $\vp>0$.

We apply Lemma \ref{Lem.perturb}, which gives the required equality.
\end{proof}
The proposition we have just proved has the following consequences.

First, the power asymptotics of $n(-t,\AF)$ as $t\to 0$ depends only on the value of the subprincipal symbol of $\AF$ at the point $x=0,$ so, we can replace this symbol by the one frozen at $x=0.$ In fact, when we make such change, we, actually, add a pseudodifferential operator of order $-1$ with symbol vanishing at $x=0$, and this change, by Proposition \ref{Prop.Est},
adds operator with fast decaying eigenvalues.
On the next step we make the homotopy of the symbol in the following way.
Continue the  symbol $\aF$ of operator $\AF_1$ from $x\in \Uc$ to $x\in \R^d,$  obtaining the symbol $\aF_2(x,\x),$ so that its principal term is smooth and zero order homogeneous in $x$ outside some neighborhood $\Uc_2\supset \Uc$ and still all eigenvalues of the principal symbol are smaller than $1-\rb.$  It is possible to perform a homotopy from $\aF_1(x,\x)$ to $\aF_2(x,\x),$ again, this homotopy not touching the symbol in $\Uc.$

Next, we note here that the value of symbol of order $-1$ at the point $x^0=0$ for our operators is the same for the Weyl quantization  and  the left one, due to relation \eqref{W-l relation}. In fact, these symbols differ by the combination of first order  $x$-derivative of the principal symbol at this point - but under our conditions these latter derivatives vanish. This means that in eigenvalue asymptotics calculations we may arbitrarily replace the Weyl  quantization by the left one and vice versa, and this does not change the form of the asymptotic coefficient.

\section{Eigenvalue asymptotics}
\subsection{The scalar case}\label{Sect.Scalar}

\emph{Proof of Theorem 1.} In the study of eigenvalue asymptotics of our operator, we consider the scalar case first, $\Nb=1$. For a given $\e,$ we consider a neighborhood of $x=0$, where
\begin{equation}\label{between}
    (1-\e)\Qb(x,\x)\le \aF_0(x,\x)\le (1+\e)\Qb(x,\x),\, x\in\Uc
\end{equation}
where $\Qb(x,\x)$ is the (scalar-valued) quadratic form \eqref{m_1}. This is possible, due to the condition \eqref{m_1}, by means of choosing a sufficiently small neighborhood $\Uc.$ We will  prove the eigenvalue asymptotics for operators
$(1\pm\e)\Q-\BF(x,\D)$ in the Weyl quantization, $\QF=\Qb_W(x,D_x)$, $\BF=\bF_W(x,D_x)$ and then use the arbitrariness of $\e$
to justify the asymptotics for $\AF-\BF$ using Lemma \ref{Lem.monoton}.

We recall that  the negative eigenvalue asymptotics does not change if we replace the subsymbol $\bF(x,\x)$ by its frozen value $\bF(0,\x).$

Consider the symbol $\cF(D)=\bF(x,D)-\bF(0,D).$
This symbol vanishes for $x=0.$ By Proposition \ref{Prop.Est}, operator $\ve \QF-\CF$ has negative eigenvalues with very fast convergence to zero, so that $n (-t, \ve \QF-\CF)=o(t^{-\vp})$ for any $\vp>0.$ By Lemma \ref{pert.estimates}, it follows that the eigenvalue asymptotics for $(1\pm\ve)\QF-\BF$ is the same as for $(1\pm\ve)\QF-\BF_0$, where $\BF_0$ is the (Weyl) pseudodifferential operator with symbol $\bF(0,\x)$ not depending on $x.$

Now let $\Fc$ be the Fourier transform in $\R^d.$
We consider operators $\LF_{\pm}=\Fc^{-1}((1\pm\ve)\QF-\BF_0)\Fc,$ unitarily equivalent to $(1\pm\ve)\QF-\BF_0)$.
Up to weaker terms we have
\begin{equation}\label{TransformedScalar}
    \LF_{\pm}=(1\pm\ve)\QF-\bF(0,x).
\end{equation}
Here $(1\pm\ve)\QF(D,x)$ is a second  order elliptic  differential operator with coefficients depending on the parameter $x$, zero order positively homogeneous in $x,$ with a smoothing near $x=0.$ In its turn, $\bF(0,x),$  is the operator of multiplication by a function $\bF(0,x)$ which decays as $h(\s)|x|^{-1}$ as $x\to \infty,$ with a smooth function $h(\s),\, \s=\frac{x}{|x|}.$
Thus, we are in the conditions of Theorem \ref{Gen.Thm.}, with $a_2(x,\x)=(1\pm\ve)\QF(\x,x)$ and $a_0(x)=\bF(0,x),$ therefore, the asymptotics of negative eigenvalues  of the operator  $(1\pm\ve)\QF(D,x)-\bF(0,x)$ is given by the formula \eqref{As.Gen.Thm.}.  The same asymptotics holds for the pre-Fourier operator $(1\pm\ve)\QF(x,D)-\bF(0,D).$ The asymptotic coefficients $\Cb$ for operators with $\pm\ve$ converge as $\ve\to 0$ to such coefficient with $\ve=0.$ Therefore, by Lemma \ref{Lem.monoton}, we have the same eigenvalue asymptotics for the operator $\aF_0(x,D)-\aF_{-1}(0,D),$ with the symbol of order $-1$ frozen at $x=0.$ Now recall that the difference $\aF_{-1}(x,D)-\aF_{-1}(0,D)$  gives a weak contribution to the negative eigenvalue asymptotics by Lemma \ref{Lem.perturb} and Proposition \ref{Prop.Est}. Finally, by  our localization considerations in Sect.4, this asymptotics holds for the initial operator $\AF.$

\subsection{The vector case}
 \emph{Proof of Theorem 2}. We consider the operator in the Euclidean space after the transformations described in Sect.\ref{Sect.Transf}. So, the lowest eigenvalue $\bm_1(x,\x)$ of the principal symbol $\aF_0(x,\x)$ has a unique and nondegenerate minimum at $x=0$ for all $\x,$ $\bm_1(0,\x)=0,$ and the complete Weyl symbol equals $(\pmb{\pmb{1}})_{\Nb\times\Nb}$ for $|x|>R.$

Using the localization property, Lemma \ref{localization}, we will change $\aF(x,\x)$ outside a neighborhood of $x=0,$ which enables further reduction to the scalar case. Let $\pb_1(x,\x)$ be the eigenprojection in $\C^\Nb$ of the symbol $\aF_0(x,\x)$ corresponding to the eigenvalue $\bm_1(x,\x).$ This eigenprojection is defined and depends smoothly on $(x,\x)$ for $x$ in a neighborhood $B_r=\{x: |x|\le r\}$ of the point $x=0,$  due to our assumption that the eigenvalue is simple in such neighborhood. We also know that $\pb_1(x,\x)$ is zero order positively homogeneous in $\x$ variable. We extend this projection-valued function as $\tilde{\pb}_1(x,\x)$ outside the ball $B_r$, so that it is smooth for all $x\in\R^d,$ has bounded derivatives in $x$ of all orders, has limit values as $|x|\to\infty$ in each direction $\frac{x}{|x|}=\s$ and positively zero order homogeneous in $\x$.

Next, consider the projection-valued function $\pb'(x,\x)= (\pmb{\pmb{1}})_{\Nb\times\Nb}-\tilde{\pb}_1(x,\x).$ For $x\in B_r,$ $\pb'(x,\x)$ is the spectral projection of $\aF_0(x,\x)$ corresponding to the eigenvalues different from $\bm_1(x,\x).$ Let $\h(s),$ $0\le s>\infty,$ be a smooth cut-off function, $0\le\h(s)\le 1 $ such that $\h(s)=1$ for $s\le r/2$ and $\h(s)=0$  for $s\ge r.$ We set \begin{equation}\label{final.eigenvalue}
    \tilde{\bm}_1(x,\x)=\h(|x|)\bm_1(x,\x)+(1-\h(|x|)),
\end{equation}
 Further on, we set $\widetilde{\bm'}(x,\x) =2r.$ Finally we set
\begin{equation}\label{final matrix}
    \tilde{\aF}(x,\x)=\tilde{\bm}_1(x,\x)\tilde{\pb}_1(x,\x)+\widetilde{\bm'}(x,\x)\pb'(x,\x).
\end{equation}
This matrix-valued function on ${T}^*\R^d$ possesses, by its construction, the following properties. For its zero order term,  $\tilde{\aF}_0(x,\x),$ lowest eigenvalue coincides with the lowest eigenvalue of $\aF_0(x,\x)$ for $x$ in a neighborhood $B_r$ of the origin; the corresponding spectral projection  of $\tilde{\aF}(x,\x)$ coincides with the spectral projection of $\aF(x,\x)$ in this neighborhood. Outside the ball, the lowest eigenvalue of $\tilde{\aF}(x,\x)$ is not greater than $1.$ The remaining eigenvalues of $\tilde{\aF}(x,\x)$ equal $2.$
As a whole, the symbol $\tilde{\aF}(x,\x)$ stabilizes in $x$ at infinity.   Therefore, the essential spectrum of the pseudodifferential operator $\widetilde{\AF},$ corresponding to the eigenvalue $\tilde{\bm}_1(x,\x)$ is separated  from the essential spectrum corresponding to other eigenvalues of the principal symbol. We can choose a closed interval $\Jc_1$ containing the essential spectrum of $\AF$ corresponding to $\tilde{\bm}_1(x,\x)$ and not touching other parts of the essential spectrum.



Now  we are in the conditions of the construction in Section 2. Using Proposition \ref{Prop.Extract}, we find a unitary pseudodifferential operator $\TF$ such that, up to an operator of order $-2,$ the part of the operator $\AF$ corresponding to the spectrum  transforms by means of $\TF$ to a scalar pseudodifferential operator. To this latter operator we apply the result on the scalar operator and obtain the eigenvalue asymptotics formula.

\section{The NP operator in elasticity.} We apply our general results to the Neumann-Poincar\'{e} operator $\KF$ in 3D elasticity. This is an operator acting on $3$-component vector functions on the two-dimensional manifold $\Xb=\partial(\Dc)$, therefore $\Nb=3$, $d=2$. Since $d=2,$ the \emph{Condition A} in Sect.2 is satisfied. As explained in the Introduction, the principal symbol $\kF_0$ of the zero order pseudodifferential operator $\KF$ is given by the expression \eqref{princ}, where $\x=(\x_1,\x_2)$ are orthogonal co-ordinates in the tangent plane to $\Xb$ (identified naturally with the cotangent plane since $\Xb$ is embedded in $\R^3$.) The symbol $\kF_0(\x)$ is represented by the matrix \eqref{princ} in the frame where two axes are directed along the $\x_1,\x_2$ axes in the tangent plane and the third axis is directed along the exterior normal  to $\Xb$.  The eigenvalues of the principal symbol equal $\bm_1(x,\x)=-\km(x),$ $\bm_2(x,\x)=0,$ $\bm_3(x,\x)=\km(x).$ We see that the eigenvalue $\bm_2(x,\x)$ is constant, it does not depend on the point $x\in\Xb$ and on $\x\in \R^d$ and thus the point zero is an isolated point of the essential spectrum of $\KF,$ independently of the geometry of $\Xb$ or the material parameters.  Therefore, the results of the present paper are not applicable to the study of eigenvalues converging to this point; the results \cite{RPc} or \cite{R.NP.} are not applicable either, since the latter papers consider the case when the whole essential spectrum consists of isolated points. Thus, an additional analysis is needed, and the corresponding results will be presented in a later publication.
 
In the opposite, the analysis of the discrete spectrum near other tips of the essential spectrum fit in the scope of the present paper.

 Suppose that the function $\km(x)$ has its nondegenerate maximal value $1$ at the point $x^0\in \Xb.$ Following our general approach, we consider the operator $1-\KF,$ for which the principal symbol $\pmb{\pmb{1}}-\kF,$ where $\pmb{\pmb{1}}$ is the $3\times 3$ unit matrix.

The eigenvector $\eb_1(x,\x)$ corresponding to the eigenvalue $\bm_1(x)$ equals
\begin{equation}\label{eb1}
    \eb_1(x,\x)=\frac{1}{\sqrt{2}}(-i\vs,1)^\top, \vs=|\x|^{-1}(\x_1,\x_2)\in S^1.
\end{equation}
This is a smooth local  eigenvector branch for the principal symbol over $S^*\Xb$ near $\x^0$, the existence of which is declared by \emph{Condition A}. Similarly,
\begin{equation}\label{eb2}
    \eb_2(x,\x)=(\vs_2,\vs_1,0)^\top,\, \eb_3(x,\x)=\frac{1}{\sqrt{2}}(i\vs,1)^\top
\end{equation}
are smooth local  eigenvector branches corresponding to the eigenvalues $\bm_2(x,\x),\,\bm_3(x,\x)$

By the construction in \cite{Capo.Diag}, there is a unitary operator $\TF$ transforming $\AF$ to a diagonal form. The principal symbol of this operator can be chosen as the matrix composed of the column vectors $\eb_\i(x,\x),$ $\i=1,2,3.$ The procedure in \cite{Capo.Diag}, see also \cite{CapoVas} describes the procedure of the construction of the consequent terms in the symbol of $\TF,$ we, however, do not need them.

The subprincipal symbol of the NP operator was calculated in \cite{R.NP.}.  To express this symbol at the point $x^0$ where the function $\km(x)$ has nondegenerate maximum, a special co-ordinate system is chosen, namely, with  $\x_1,\x_2$ directed along the  principal curvature directions of the surface $\Xb$ (if $x^0$ is an umbilical point, these directions are chosen arbitrarily); for definiteness, these directions are chosen to form a right co-ordinate system.

In such co-ordinates and the corresponding frame, it is found in \cite{R.NP.} that the subprincipal symbol is a linear combination of the principal curvatures $\k_1(x^0),$ $\k_2(x_0)$ of $\Xb,$ with coefficients being universal (not depending on the geometry of $\Xb$) order $-1$ symbols,  $3\times 3$ matrices with entries depending linear fractionally  on the material constants.
(We do not copy the explicit expressions here.) 

Applying our general results, we arrive at the asymptotics of eigenvalues of the NP operator in 3D elasticity. Namely, the counting function of the eigenvalues approaching the nondegenerate maximum point $x^0$ of the coefficient $\km(x)$ at the boundary, $\km(x^0)=1,$ has the order $n(1+t, \KF)\sim C t^{-1}$ (here $\frac{d}{2}=1$), where the coefficient $C$ depends on the principal curvatures of $\Xb$ at the point $x^0$ and material constants at this point.

When determining the asymptotics of eigenvalues converging to other extremal points of $\km$ and $-\km,$ one can make the linear fractional transform of the operator $\KF$ (as it was done, e.g., in \cite{MiRo},) so that this value becomes the maximal value of the principal symbol of the transformed operator.

In comparison with the case of a homogeneous material, investigated in \cite{MiRo}, \cite{R.NP.}, where the asymptotics under discussion was found to have the order $t^{-2}$,  we have, for the nonhomogeneous material, a faster eigenvalue convergence to the tip of the essential spectrum.

\end{document}